\documentclass [12pt]{amsart}
\usepackage{amssymb,amsmath,amsthm}
\newtheorem{theorem}{Theorem}
\newtheorem{lemma}{Lemma}

\newcommand{\ad}{\,\mathrm{ad}\,}

\newcommand{\GL}{\,\mathrm{GL}\,}
\newcommand{\SL}{\,\mathrm{SL}\,}

\begin{document}

\begin{center}

{\Large {\bf Automorphisms of Chevalley groups

of types  $A_l, D_l, E_l$ over local rings with $1/2$\footnote{The
work is supported by the Russian President grant MK-2530.2008.1 and
by the grant of Russian Fond of Basic Research 08-01-00693.} }}

\bigskip
\bigskip

{\large \bf E.~I.~Bunina}

\end{center}
\bigskip

\begin{center}

{\bf Abstract.}

\end{center}

In this paper we prove that every automorphism of a (elementary)
Chevalley group of type $A_l, D_l$, or $E_l$, $l\geqslant 2$, over
a commutative local ring with $1/2$ is standard, i.\,e., is
the composition of inner, ring, graph and central automorphisms.

\bigskip

\section*{Introduction}\leavevmode

Let $G_{\ad}$ be a Chevalley-Demazure group scheme associated with
an irreducible root system~$\Phi$ of rank $>1$, $G_{\pi}(\Phi,R)$ be
a set of points~$G_{\pi}$ with values in â~$R$;
$E_{\pi}(\Phi,R)$ be an elementary subgroup in~$G_{\pi}(\Phi,R)$,
where $R$
is a commutative ring with a unit. In the given work we describe automorphisms
of the groups $E_{\pi}(\Phi,R)$ and $G_\pi (\Phi,R)$ over local
commutative rings  with~$1/2$, for root systems of types $A_l,D_l,E_l$. Similar results for Chevalley groups over fields were proved
 by R.\,Steinberg~\cite{Stb1} for the finite case and by J.\,Humphreys~\cite{H} for the infinite case. Many papers were devoted
to description of automorphisms of Chevalley groups over different
commutative rings, we can mention here the papers of
Borel--Tits~\cite{v22}, Carter--Chen~Yu~\cite{v24},
Chen~Yu~\cite{v25}--\cite{v29}, E.\,Abe~\cite{Abe_OSN},
A.\,Klyachko~\cite{Klyachko}.

In the papers \cite{Bun1}, \cite{Bun2} the author showed that automorphisms of adjoint elementary Chevalley groups (with roots systems under consideration)
over local rings with~$1/2$ can be represented as the composition of ring and  \emph{automorphisms--conjugation}, where  automorphism--conjugation is a conjugation of the Chevalley group in the adjoint representation with some matrix from the normalizer of this group in $\GL(V)$.

Using results of the papers \cite{Bun1},~\cite{Bun2}, here we can describe  automorphisms  of (elementary) Chevalley groups
of ranks $>1$ over arbitrary commutative rings with $1/2$ and root systems $A_l,D_l,E_l$ (prove that they are standard). By standard automorphisms we mean here compositions of inner, ring, graph and central automorphisms.

To proof the main theorem we describe normalizers of adjoint elementary Chevalley groups in the adjoint representation. Note that the normalizer of the simply connected Chevalley group of type $E_6$ in its $27$-dimensional representation is described by Vavilov and Luzgarev in~\cite{VavLuzg}.

The author is thankful to N.A.\,Vavilov,  A.A.\,Klyachko,
A.V.\,Mikhalev for valuable advices, remarks and discussions.

\section{Definitions and main theorems.}\leavevmode

We fix a root system~$\Phi$ of rank $>1$. Detailed texts about root
systems and their properties can be found in the books
\cite{Hamfris}, \cite{Burbaki}. Suppose now that we have a
semisimple complex Lie algebra~$\mathcal L$ of type $\Phi$ with
Cartan subalgebra~$\mathcal H$ (detailed information about
semisimple Lie algebras can be found in the book~\cite{Hamfris}).

We can choose a basis $\{ h_1, \dots, h_l\}$ in~$\mathcal H$ and for
every $\alpha\in \Phi$ elements $x_\alpha \in {\mathcal L}_\alpha$
so that $\{ h_i; x_\alpha\}$ is a basis in~$\mathcal L$ and for
every two elements of this basis their commutator is an integral
linear combination of the elements of the same basis.

Introduce now elementary Chevalley groups (see~\cite{Steinberg}).

Let  $\mathcal L$ be a semisimple Lie algebra (over~$\mathbb C$)
with a root system~$\Phi$, $\pi: {\mathcal L}\to gl(V)$ be its
finitely dimensional faithful representation  (of dimension~$n$). If
$\mathcal H$ is a Cartan subalgebra of~$\mathcal L$, then a
functional
 $\lambda \in {\mathcal H}^*$ is called a
 \emph{weight} of  a given representation, if there exists a nonzero vector $v\in V$
 (that is called a  \emph{weight vector}) such that
for any $h\in {\mathcal H}$ $\pi(h) v=\lambda (h)v.$

In the space~$V$ there exists a basis of weight vectors such that
all operators $\pi(x_\alpha)^k/k!$ for $k\in \mathbb N$ are written
as integral (nilpotent) matrices. This basis is called a
\emph{Chevalley basis}. An integral matrix also can be considered as
a matrix over an arbitrary commutative ring with~$1$. Let $R$ be
such a ring. Consider matrices $n\times n$ over~$R$, matrices
$\pi(x_\alpha)^k/k!$ for
 $\alpha\in \Phi$, $k\in \mathbb N$ are included in $M_n(R)$.

Now consider automorphisms of the free module $R^n$ of the form
$$
\exp (tx_\alpha)=x_\alpha(t)=1+t\pi(x_\alpha)+t^2
\pi(x_\alpha)^2/2+\dots+ t^k \pi(x_\alpha)^k/k!+\dots
$$
Since all matrices $\pi(x_\alpha)$ are nilpotent, we have that this
series is finite. Automorphisms $x_\alpha(t)$ are called
\emph{elementary root elements}. The subgroup in $Aut(R^n)$,
generated by all $x_\alpha(t)$, $\alpha\in \Phi$, $t\in R$, is
called an \emph{elementary Chevalley group} (notation:
$E_\pi(\Phi,R)$).

The action of elements $x_\alpha(t)$ on the Chevalley basis is described in
\cite{v23}, \cite{VavPlotk1}.

All weights of a given representation (by addition) generate a
lattice (free Abelian group, where every  $\mathbb Z$-basis  is also
a $\mathbb C$-basis in~${\mathcal H}^*$), that is called the
\emph{weight lattice} $\Lambda_\pi$.

Elementary Chevalley groups are defined not even by a representation
of the Chevalley groups, but just by its \emph{weight lattice}.
Namely, up to an abstract isomorphism an elementary Chevalley group
is completely defined by a root system~$\Phi$, a commutative
ring~$R$ with~$1$ and a weight lattice~$\Lambda_\pi$.

Among all lattices we can mark  the lattice corresponding to the
adjoint representation: it is generated by all roots (the \emph{root
lattice}~$\Lambda_{ad}$). The corresponding (elementary) Chevalley group is called \emph{adjoint}.

Introduce now Chevalley groups (see~\cite{Steinberg},
\cite{Chevalley}, \cite{v3}, \cite{v23}, \cite{v30}, \cite{v43},
\cite{VavPlotk1}, and also latter references in these papers).

Consider semisimple linear algebraic groups over algebraically
closed fields. These are precisely elementary Chevalley groups
$E_\pi(\Phi,K)$ (see.~\cite{Steinberg},~\S\,5).

All these groups are defined in $SL_n(K)$ as  common set of zeros of
polynomials of matrix entries $a_{ij}$ with integer coefficients
 (for example,
in the case of the root system $C_l$ and the universal
representation we have $n=2l$ and the polynomials from the condition
$(a_{ij})Q(a_{ji})-Q=0$). It is clear now that multiplication and
taking inverse element are also defined by polynomials with integer
coefficients. Therefore, these polynomials can be considered as
polynomials over arbitrary commutative ring with a unit. Let some
elementary Chevalley group $E$ over~$\mathbb C$ be defined in
$SL_n(\mathbb C)$ by polynomials $p_1(a_{ij}),\dots, p_m(a_{ij})$.
For a commutative ring~$R$ with a unit let us consider the group
$$
G(R)=\{ (a_{ij})\in \SL_n(R)\mid \widetilde p_1(a_{ij})=0,\dots
,\widetilde p_m(a_{ij})=0\},
$$
where  $\widetilde p_1(\dots),\dots \widetilde p_m(\dots)$ are
polynomials having the same coefficients as
$p_1(\dots),\dots,p_m(\dots)$, but considered over~$R$.

This group is called the \emph{Chevalley group} $G_\pi(\Phi,R)$ of
the type~$\Phi$ over the ring~$R$, and for every algebraically
closed field~$K$ it coincides with the elementary Chevalley group.

The subgroup of diagonal (in the standard basis of weight vectors)
matrices of the Chevalley group $G_\pi(\Phi,R)$ is called the
 \emph{standard maximal torus}
of $G_\pi(\Phi,R)$ and it is denoted by $T_\pi(\Phi,R)$. This group
is isomorphic to $Hom(\Lambda_\pi, R^*)$.

Let us denote by $h(\chi)$ the elements of the torus $T_\pi
(\Phi,R)$, corresponding to the homomorphism $\chi\in Hom
(\Lambda(\pi),R^*)$.

In particular, $h_\alpha(u)=h(\chi_{\alpha,u})$ ($u\in R^*$, $\alpha
\in \Phi$), where
$$
\chi_{\alpha,u}: \lambda\mapsto u^{\langle
\lambda,\alpha\rangle}\quad (\lambda\in \Lambda_\pi).
$$

Note that the condition
$$
G_\pi (\Phi,R)=E_\pi (\Phi,R)
$$
is not true even for fields, that are not algebraically closed.
Let us show the difference between Chevalley groups and their elementary subgroups in the case when $R$ is semilocal. In this case
 $G_\pi (\Phi,R)=E_\pi(\Phi,R)T_\pi(\Phi,R)$
(see~\cite{v38}, \cite{Abe1}, \cite{v19}), and elements $h(\chi)$ are connected with elementary generators
by the formula
\begin{equation}\label{ee4}
h(\chi)x_\beta (\xi)h(\chi)^{-1}=x_\beta (\chi(\beta)\xi).
\end{equation}

 In the case of semilocal rings from  the formula~\eqref{ee4} we see that
$$
[G(\Phi,R),G(\Phi,R)]\subseteq E(\Phi,R).
$$
If $R$ (as in our case) also contains~$1/2$,
$l\geqslant 2$, we can easily show that
$$
[G(\Phi,R),G(\Phi,R)]=[E(\Phi,R), E(\Phi,R)]=E(\Phi,R).
$$

Define four types of automorphisms of a Chevalley group
 $G_\pi(\Phi,R)$, we
call them  \emph{standard}.

{\bf Central automorphisms.} Let $C_G(R)$ be a center of
$G_\pi(\Phi,R)$, $\tau: G_\pi(\Phi,R) \to C_G(R)$ be some
homomorphism of groups. Then the mapping $x\mapsto \tau(x)x$ from
$G_\pi(\Phi,R)$ onto itself is an automorphism of $G_\pi(\Phi,R)$,
that is denoted by~$\tau$ and called a \emph{central automorphism}
of the group~$G_\pi(\Phi,R)$.

{\bf Ring automorphisms.} Let $\rho: R\to R$ be an automorphism of
the ring~$R$. The mapping $x\mapsto \rho (x)$ from $G_\pi(\Phi,R)$
onto itself is an automorphism of the group $G_\pi(\Phi,R)$, that is
denoted by the same letter~$\rho$ and is called a \emph{ring
automorphism} of the group~$G_\pi(\Phi,R)$. Note that for all
$\alpha\in \Phi$ and $t\in R$ an element $x_\alpha(t)$ is mapped to
$x_\alpha(\rho(t))$.

{\bf Inner automorphisms.} Let $S$ be some ring containing~$R$,  $g$
be an element of $G_\pi(\Phi,S)$, that normalizes the subgroup $G_\pi(\Phi,R)$. Then
the mapping $x\mapsto gxg^{-1}$  is an automorphism
of the group~$G_\pi(\Phi,R)$, that is denoted by $i_g$ and is called an
\emph{inner automorphism}, \emph{induced by the element}~$g\in G_\pi(\Phi,S)$. If $g\in G_\pi(\Phi,R)$, then call $i_g$ a \emph{strictly inner}
automorphism.

{\bf Graph automorphisms.} Let $\delta$ be an automorphism of the
root system~$\Phi$ such that $\delta \Delta=\Delta$. Then there
exists a unique automorphisms of $G_\pi (\Phi,R)$ (we denote it by
the same letter~$\delta$) such that for every $\alpha \in \Phi$ and
$t\in R$ an element $x_\alpha (t)$ is mapped to
$x_{\delta(\alpha)}(\varepsilon(\alpha)t)$, where
$\varepsilon(\alpha)=\pm 1$ for all $\alpha \in \Phi$ and
$\varepsilon(\alpha)=1$ for all $\alpha\in \Delta$.

 Similarly we can define four type of automorphisms of the elementary
subgroup~$E(R)$.
An automorphism~$\sigma$ of the group
 $G_\pi(\Phi,R)$ (or $E_\pi(\Phi,R)$)
is called  \emph{standard} if it is a composition of automorphisms
of these introduced four types.

Together with standard automorphisms we use the following ''temporary'' type of automorphisms of an elementary adjoint Chevalley group:

{\bf Automorphisms--conjugations.} Let $V$ be the representation space of the group $E_{\ad} (\Phi,R)$, $C\in \GL(V)$ be some matrix from its normalizer:
$$
C E_{\ad}(\Phi,R) C^{-1}= E_{\ad} (\Phi,R).
$$
 Then the mapping
 $x\mapsto CxC^{-1}$ from $E_\pi(\Phi,R)$ onto itself is an automorphism of the Chevalley group, that is denoted by~$i_Ñ$ and is called an  \emph{automorphism--conjugation} of~$E(R)$,
\emph{induced by an element}~$C$ of~$\GL(V)$.

Our aim is to prove the next main theorem:

\begin{theorem}\label{main}
Let $G=G_{\pi}(\Phi,R)$ $(E_\pi(\Phi,R))$
be an \emph{(}elementary\emph{)} Chevalley group with a the root system $A_l, D_l$, or $E_l$, $l\geqslant 2$, $R$ is a commutative local ring
with~$1/2$. Then every automorphism of~$G$ is standard. If the Chevalley group is adjoint, then an inner automorphism in the composition is strictly
inner.
\end{theorem}

To prove this theorem we use the main theorem of the paper~\cite{Bun2}:

\begin{theorem}\label{old}
Every automorphism of an elementary adjoint Chevalley group of type
$A_l,D_l$, or $E_l$, $l\geqslant 2$, over a local ring with~$1/2$
is a composition of a ring automorphism and an automorphism--conjugation.
\end{theorem}

To apply Theorem~\ref{old}, we need to prove the following important (having also a proper interest) fact:

\begin{theorem}\label{norm}
Every  automorphism--conjugation of an elementary adjoint Chevalley group  of type $A_l, D_l$, or $E_l$, $l\geqslant 2$, over a local ring is a composition of a strictly inner \emph{(}conjugation with the help of the corresponding Chevalley group\emph{)} and graph automorphisms.
\end{theorem}

Three next sections are devoted to the proof of Theorem~\ref{norm}.

\section{Reduction to systems of linear equations over local rings.}\leavevmode

In this section we consider an elementary adjoint Chevalley group
$E_{\ad}(\Phi,R)$ of types $A_l, D_l$, or $E_l$, $l\geqslant 2$, in
its adjoint $n$-dimensional representation, $R$ as an arbitrary
local ring.

Let $n=l+2m$, where $m$ is the number of positive roots of the
system~$\Phi$, basic vectors are numerated as
 $v_1=x_{\alpha_1}, v_{-1}=x_{-\alpha_1}, \dots,
v_m=x_{\alpha_m}, v_{-m}=x_{-\alpha_m}, V_1=h_{1},\dots,V_l=h_{l}$,
it respects to the Chevalley basis of~$\Phi$. For our convenience we
denote matrix units as $e_{\alpha,\beta}$, $e_{\alpha,h_i}$,
$e_{h_i,\alpha}$, $e_{h_i,h_j}$, $\alpha,\beta\in \Phi$, $1\leqslant
i,j\leqslant l$.

Suppose that we have some matrix $C=(c_{i,j})\in \GL_n(R)$ such that
$$
C\cdot E_{\ad}(\Phi,R) \cdot C^{-1}=E_{\ad}(\Phi,R).
$$

If $J$ is the maximal ideal (radical) of $R$, then the matrices from
$M_n(J)$ form the radical of the matrix ring $M_n(R)$, therefore
$$
C\cdot M_n(J)\cdot C^{-1}=M_n(J),
$$
consequently
$$
C\cdot (E+M_n(J))\cdot C^{-1}=E+M_n(J),
$$
so
$$
C\cdot E_{\ad}(\Phi,R,J)\cdot C^{-1}=E_{\ad}(\Phi,R,J),
$$
since $E_{\ad}(\Phi,R,J)=E_{\ad}(\Phi,R)\cap (E+M_n(J)).$

Therefore the image $\overline C$ of~$C$ under factorization of the
ring $R$ by its radical~$J$ gives us an   automorphism--conjugation of
the Chevalley group $E_{\ad} (\Phi,k)$, where $k=R/J$ is a residue
field of~$R$.

\begin{lemma}\label{fromAnton}
If $E_{\ad}(\Phi,k)$ is a Chevalley group of type $A_l$, $D_l$, or $E_l$, $l>1$, over a field $k$, then every its automorphism--conjugation  is a composition of inner and graph automorphisms.
\end{lemma}
\begin{proof}
By Theorem~30 from~\cite{Steinberg} every automorphism of a Chevalley group with any root system under consideration over a field~$k$ is standard, i.\,e.,  it is a composition of inner, ring and graph automorphisms. Suppose that a matrix~$C$ is from normalizer of $E_{\ad}(\Phi,k)$ in $\GL_{n}(k)$. Then $i_C$ is an automorphism of $E_{\ad}(\Phi,k)$, so we have $i_C=i_g\circ \delta \circ \rho$, $g\in E_{\ad}(\Phi,k)$, $\delta$ is a graph automorphism, $\rho$ is a ring automorphism. Taking a matrix $A_\delta=\sum_{\alpha\in \Phi} \pm e_{\delta(\alpha),\alpha}+\sum_{\alpha\in \Delta} \pm e_{H_{\delta(\alpha)}, H_\alpha}$, we induce a graph automorphism $\delta$. Consequently, $i_{A_{\delta}^{-1}} i_{g^{-1}}i_C=i_{C'}=\rho$ and some matrix $C'\in \GL_{n}(k)$ defines a ring automorphism~$\rho$. For every root $\alpha\in \Phi$ we have $\rho(x_\alpha(1))=x_\alpha(1)$, therefore $C'x_\alpha(1)=x_\alpha(1) C'$ for all $\alpha\in \Phi$. Thus we have that $C'$ is scalar and an automorphism $i_C$ is a composition of inner and graph automorphisms.
\end{proof}

By Lemma~\ref{fromAnton}
$$
i_{\overline C}=i_{A_\delta}i_g,\quad  g\in E_{\ad}(\Phi,k).
$$

The matrix $A_\delta$ consists of $0$ and  $\pm 1$, therefore it can be
considered as a matrix from the group $\GL_n(R)$.

Since over a field every element of the Chevalley group is a product of some element from elementary subgroup (that is generated by unipotents $x_\alpha(t)$)
and some torus element, then the matrix $g$ can be decomposed to the product
$t_{\alpha_1}(X_1)\dots t_{\alpha_l}(X_l) x_{\alpha_{i_1}}(Y_1)\dots
x_{i_N}(Y_N)$, where $X_1,\dots,X_l,Y_1,\dots, Y_N\in k$,
$t_{\alpha_k}(X)$ is a torus element, corresponding to a
homomorphism~$\chi$ such that $\chi(\alpha_k)=X$, $\chi(\alpha_j)=1$
for $j\ne k$, $1\leqslant j\leqslant l$.

Since every element $X_1,\dots, X_l, Y_1,\dots, Y_N$ is a residue
class in the ring~$R$, we can choose (by arbitrary way) elements
$x_1\in X_1$, \dots, $x_l\in X_l$, $y_1\in Y_1$, \dots, $y_N\in
Y_n$, and the element
$$
g'=t_{\alpha_1}(x_1)\dots t_{\alpha_l}(x_l)
x_{\alpha_{i_1}}(y_1)\dots x_{i_N}(y_N)
$$
satisfies the conditions $g'\in G_{\ad}(\Phi,R)$ and $\overline
{g'}=g$.

Consider the matrix $C'={g'}^{-1}\circ {A_\delta}^{-1}\circ C$. This matrix
also normalizes an elementary Chevalley group $E_{\ad}(\Phi,R)$,
moreover $\overline {C'}=E$. Therefore we reduce the description of
matrices from the normalizer of $E_{\ad}(\Phi,R)$ to the description
of matrices from this normalizer, equivalent to~$E$ modulo~$J$.

Consequently we can suppose that our initial matrix $C$ is
equivalent to~$E$ modulo~$J$.

We want to show that $C\in G_{\ad}(\Phi,R)$.

At first we want to prove one technical lemma.

\begin{lemma}\label{prod2}
Let $X=\lambda t_{\alpha_1}(s_1)\dots t_{\alpha_l}(s_l)x_{\alpha_1}(t_1)\dots x_{\alpha_m}(t_m)x_{-\alpha_1}(u_1)\dots x_{-\alpha_m}(u_m)\in \lambda E_{\ad}(\Phi,R,J)$.
Then the matrix $X$ has such $n+1$ coefficients \emph{(}precisely described in the proof of this lemma\emph{)} that uniquely define elements $\lambda$, $s_1,s_2$, $t_1,\dots,t_m$, $u_1,\dots, u_m$.
\end{lemma}

\begin{proof}
At first we show the matrix $X$ in the case of the root system $A_2$:
\begin{multline*}
X=\lambda t_{\alpha_1}(s_1)t_{\alpha_2}(s_2)\times\\
\times x_{\alpha_1}(t_1)x_{\alpha_2}(t_2)x_{\alpha_1+\alpha_2}(t_3)x_{-\alpha_1}(u_1)x_{-\alpha_2}(u_2)x_{-\alpha_1-\alpha_2}(u_3)=\\
=\begin{pmatrix}
*& * & * & * & * & * & * & *\\
*& \frac{\lambda (1-t_2u_2)}{s_1} & * & * & * & \frac{\lambda t_2}{s_1}& * & *\\
*& * & * & * & * & * & * & *\\
*& * & * & \frac{\lambda (1-t_1u_1)}{s_2} & * &  -\frac{ \lambda t_1}{s_2}& * & *\\
*& * & * & * & * & * & * & *\\
*& -\frac{\lambda u_2}{s_1s_2} & * & \frac{\lambda u_1}{s_1 s_2} & * & \frac{\lambda}{s_1 s_2}& * & -\frac{\lambda(u_3+2u_1u_2)}{s_1s_2}\\
*& * & * & * & * & * & * & *\\
*& * & * & * & * & \lambda t_3& * & *
\end{pmatrix}.
\end{multline*}
It is clear that all numbers $\lambda, s_1,s_2,t_1,t_2,t_3,u_1,u_2,u_3$ are uniquely defined but a given matrix~$X$.

Now we come to an arbitrary root system. To do it we need for every system under consideration mark special sequences of roots.

\emph{Marked roots of the system $A_l$.}

Suppose that the root system $A_l$ has simple roots $\alpha_1,\dots, \alpha_l$. Its maximal root is $\beta_{1,l}=\alpha_1+\dots+\alpha_l$,
every root has the form $\beta_{i,j}=\alpha_i+\alpha_{i+1}+\dots+\alpha_{j-1}+\alpha_j$, where $1\leqslant i\leqslant j\leqslant l$.
In a given system we consider the sequence of roots $\beta_{1,l},\beta_{2,l},\dots, \beta_{l-1,l}, \beta_{l}$. Every root of $A_l$ is
a difference of two distinct roots of our sequence (or the member of this sequence).

\emph{Marked roots of the system $D_l$.}

Roots of the system $D_l$ have the form $\pm e_i \pm e_j$, $1\leqslant i < j \leqslant l$, $\{ e_1,\dots, e_l\}$ is an ortonormal basis of the space~$\mathbb R^l$.
Simple roots are $\alpha_1=e_1-e_2, \alpha_2=e_2-e_3,\dots, \alpha_{l-2}=e_{l-2}-e_{l-1}, \alpha_{l-1}=e_{l-1}-e_{l}, \alpha_l=e_{l-1}-e_l$.
The maximal root is $e_1+e_2=\alpha_1+2\alpha_2+2\alpha_3+\dots+2\alpha_{l-2}+\alpha_{l-1}+\alpha_l$. In a given root system we consider the
sequence
{\footnotesize
\begin{align*}
&e_1+e_2= \alpha_1+2\alpha_2+2\alpha_3+\dots+2\alpha_{l-2}+\alpha_{l-1}+\alpha_l,\\
&e_1+e_3=\alpha_1+\alpha_2+2\alpha_3+\dots+2\alpha_{l-2}+\alpha_{l-1}+\alpha_l,\\
&e_2+e_3=\alpha_2+2\alpha_3+\dots+2\alpha_{l-2}+\alpha_{l-1}+\alpha_l,\\
&e_2+e_4=\alpha_2+\alpha_3+\dots+2\alpha_{l-2}+\alpha_{l-1}+\alpha_l,\\
 &\dots\ldots\ldots\dots\dots\dots,\\
&e_2+e_{l-2}=\alpha_2+\dots+\alpha_{l-3}+2\alpha_{l-2}+\alpha_{l-1}+\alpha_l,\\
&e_2+e_{l-1}=\alpha_2+\dots+\alpha_{l-3}+\alpha_{l-2}+\alpha_{l-1}+\alpha_l,\\
&e_2-e_l=
\alpha_2+\dots+\alpha_{l-2}+\alpha_{l-1},\\
&e_2-e_{l-1}=\alpha_2+\dots+\alpha_{l-2},\\
&\dots\ldots\ldots\dots\dots\dots,\\
&e_2-e_3=\alpha_2.
\end{align*}
}

Every following element of the sequence is obtained from the previous one by subtracting some simple root. Also every root of $D_l$ is either a member of the sequence, or the difference of two roots from the sequence under consideration ($e_i-e_j=(e_1+e_i)-(e_1+e_j)$, $e_i+e_j=(e_1+e_i)-(e_1-e_j)$).

\emph{Marked roots of the system $E_l$.}

Since the root systems  $E_6$ and $E_7$ are embedded in the system $E_8$, for convenience we will consider namely this system.

 Roots of the system $E_8$ have the form $\pm e_i\pm e_j$, $1\leqslant i< j\leqslant 8$ and $\frac{1}{2}(\pm e_1 \pm e_2 \pm e_3 \pm e_4 \pm e_5 \pm e_6 \pm e_7 \pm e_8)$, the number of minuses is even.
Simple roots: $\alpha_1=e_1-e_2$, $\alpha_2=\frac{1}{2}(-e_1-e_2-e_3+e_4+e_5+e_6+e_7+e_8)$, $\alpha_3=e_2-e_3$, $\alpha_4=e_3-e_4$, $\alpha_5=e_4-e_5$, $\alpha_6=e_5-e_6$, $\alpha_7=e_6-e_7$, $\alpha_8=e_7-e_8$. The maximal root $\frac{1}{2}(e_1+e_2+e_3+e_4+e_5+e_6+e_7-e_8)=2\alpha_1+3\alpha_2+4\alpha_3+6\alpha_4+5\alpha_5+4\alpha_6+3\alpha_7+2\alpha_8$.
We consider the next system
{\scriptsize
\begin{align*}
&\frac{1}{2}(e_1+e_2+e_3+e_4+e_5+e_6+e_7-e_8)=2\alpha_1+3\alpha_2+4\alpha_3+6\alpha_4+5\alpha_5+4\alpha_6+3\alpha_7+2\alpha_8,\\
&\frac{1}{2}(e_1+e_2+e_3+e_4+e_5+e_6-e_7+e_8)=2\alpha_1+3\alpha_2+4\alpha_3+6\alpha_4+5\alpha_5+4\alpha_6+3\alpha_7+\alpha_8,\\
&\frac{1}{2}(e_1+e_2+e_3+e_4+e_5-e_6+e_7+e_8)=2\alpha_1+3\alpha_2+4\alpha_3+6\alpha_4+5\alpha_5+4\alpha_6+2\alpha_7+\alpha_8,\\
&\frac{1}{2}(e_1+e_2+e_3+e_4-e_5+e_6+e_7+e_8)=2\alpha_1+3\alpha_2+4\alpha_3+6\alpha_4+5\alpha_5+3\alpha_6+2\alpha_7+\alpha_8,\\
&\frac{1}{2}(e_1+e_2+e_3-e_4+e_5+e_6+e_7+e_8)=2\alpha_1+3\alpha_2+4\alpha_3+6\alpha_4+4\alpha_5+3\alpha_6+2\alpha_7+\alpha_8,\\
&\frac{1}{2}(e_1+e_2-e_3+e_4+e_5+e_6+e_7+e_8)=2\alpha_1+3\alpha_2+4\alpha_3+5\alpha_4+4\alpha_5+3\alpha_6+2\alpha_7+\alpha_8,\\
&e_1+e_2=2\alpha_1+2\alpha_2+4\alpha_3+5\alpha_4+4\alpha_5+3\alpha_6+2\alpha_7+\alpha_8,\\
&e_1+e_3=2\alpha_1+2\alpha_2+3\alpha_3+5\alpha_4+4\alpha_5+3\alpha_6+2\alpha_7+\alpha_8,\\
&e_2+e_3=\alpha_1+2\alpha_2+3\alpha_3+5\alpha_4+4\alpha_5+3\alpha_6+2\alpha_7+\alpha_8,\\
&e_2+e_4=\alpha_1+2\alpha_2+3\alpha_3+4\alpha_4+4\alpha_5+3\alpha_6+2\alpha_7+\alpha_8,\\
&e_2+e_5=\alpha_1+2\alpha_2+3\alpha_3+4\alpha_4+3\alpha_5+3\alpha_6+2\alpha_7+\alpha_8,\\
&e_3+e_5=\alpha_1+2\alpha_2+2\alpha_3+4\alpha_4+3\alpha_5+3\alpha_6+2\alpha_7+\alpha_8,\\
&e_4+e_5=\alpha_1+2\alpha_2+2\alpha_3+3\alpha_4+3\alpha_5+3\alpha_6+2\alpha_7+\alpha_8,
\end{align*}

\begin{align*}
&\frac{1}{2}(e_1+e_2+e_3+e_4+e_5-e_6-e_7-e_8)=\alpha_1+\alpha_2+2\alpha_3+3\alpha_4+3\alpha_5+3\alpha_6+2\alpha_7+\alpha_8,\\
&\frac{1}{2}(e_1+e_2+e_3+e_4-e_5+e_6-e_7-e_8)=\alpha_1+\alpha_2+2\alpha_3+3\alpha_4+3\alpha_5+2\alpha_6+2\alpha_7+\alpha_8,\\
&\frac{1}{2}(e_1+e_2+e_3-e_4+e_5+e_6-e_7-e_8)=\alpha_1+\alpha_2+2\alpha_3+3\alpha_4+2\alpha_5+2\alpha_6+2\alpha_7+\alpha_8,\\
&\frac{1}{2}(e_1+e_2-e_3+e_4+e_5+e_6-e_7-e_8)=\alpha_1+\alpha_2+2\alpha_3+2\alpha_4+2\alpha_5+2\alpha_6+2\alpha_7+\alpha_8,\\
&\frac{1}{2}(e_1-e_2+e_3+e_4+e_5+e_6-e_7-e_8)=\alpha_1+\alpha_2+\alpha_3+2\alpha_4+2\alpha_5+2\alpha_6+2\alpha_7+\alpha_8,\\
&\frac{1}{2}(e_1-e_2+e_3+e_4+e_5-e_6+e_7-e_8)=\alpha_1+\alpha_2+\alpha_3+2\alpha_4+2\alpha_5+2\alpha_6+\alpha_7+\alpha_8,\\
&\frac{1}{2}(e_1-e_2+e_3+e_4-e_5+e_6+e_7-e_8)=\alpha_1+\alpha_2+\alpha_3+2\alpha_4+2\alpha_5+\alpha_6+\alpha_7+\alpha_8,\\
&\frac{1}{2}(e_1-e_2+e_3-e_4+e_5+e_6+e_7-e_8)=\alpha_1+\alpha_2+\alpha_3+2\alpha_4+\alpha_5+\alpha_6+\alpha_7+\alpha_8,\\
&\frac{1}{2}(e_1-e_2-e_3+e_4+e_5+e_6+e_7-e_8)=\alpha_1+\alpha_2+\alpha_3+\alpha_4+\alpha_5+\alpha_6+\alpha_7+\alpha_8,\\
&e_1-e_8=\alpha_1+\alpha_3+\alpha_4+\alpha_5+\alpha_6+\alpha_7+\alpha_8,\\
&e_1-e_7=\alpha_1+\alpha_3+\alpha_4+\alpha_5+\alpha_6+\alpha_7,\\
&e_1-e_6=\alpha_1+\alpha_3+\alpha_4+\alpha_5+\alpha_6,\\
&e_1-e_5=\alpha_1+\alpha_3+\alpha_4+\alpha_5,\\
&e_1-e_4=\alpha_1+\alpha_3+\alpha_4,\\
&e_1-e_3=\alpha_1+\alpha_3,\\
&e_1-e_2=\alpha_1.
\end{align*}
}

Every root of $E_8$ is either a member of the sequence, or is the difference between some roots of thence (it is checked by direct calculations), except the following roots:
{\footnotesize
\begin{align*}
&\frac{1}{2}(e_1-e_2-e_3+e_4+e_5+e_6-e_7+e_8)=\alpha_1+\alpha_2+\alpha_3+\alpha_4+\alpha_5+\alpha_6+\alpha_7,\\
&\frac{1}{2}(e_1-e_2+e_3-e_4+e_5+e_6-e_7+e_8)=\alpha_1+\alpha_2+\alpha_3+2\alpha_4+\alpha_5+\alpha_6+\alpha_7,\\
&\frac{1}{2}(e_1-e_2+e_3+e_4+e_5-e_6-e_7+e_8)=\alpha_1+\alpha_2+\alpha_3+2\alpha_4+2\alpha_5+2\alpha_6+\alpha_7,\\
&\frac{1}{2}(e_1+e_2-e_3-e_4+e_5+e_6+e_7-e_8)=\alpha_1+\alpha_2+2\alpha_3+2\alpha_4+\alpha_5+\alpha_6+\alpha_7+\alpha_8,\\
&\frac{1}{2}(e_1+e_2-e_3+e_4+e_5-e_6+e_7-e_8)=\alpha_1+\alpha_2+2\alpha_3+2\alpha_4+2\alpha_5+2\alpha_6+\alpha_7+\alpha_8,\\
&\frac{1}{2}(e_1+e_2+e_3-e_4+e_5-e_6+e_7-e_8)=\alpha_1+\alpha_2+2\alpha_3+3\alpha_4+2\alpha_5+2\alpha_6+\alpha_7+\alpha_8,\\
&\frac{1}{2}(e_1+e_2+e_3+e_4-e_5-e_6+e_7-e_8)=\alpha_1+\alpha_2+2\alpha_3+3\alpha_4+3\alpha_5+2\alpha_6+\alpha_7+\alpha_8,\\
&e_1+e_8=2\alpha_1+2\alpha_2+3\alpha_3+4\alpha_4+3\alpha_5+2\alpha_6+\alpha_7,\\
&e_1+e_7=2\alpha_1+2\alpha_2+3\alpha_3+4\alpha_4+3\alpha_5+2\alpha_6+\alpha_7+\alpha_8,\\
&e_1+e_6=2\alpha_1+2\alpha_2+3\alpha_3+4\alpha_4+3\alpha_5+2\alpha_6+2\alpha_7+\alpha_8,\\
&\frac{1}{2}(e_1-e_2+e_3+e_4+e_5+e_6+e_7+e_8)=2\alpha_1+3\alpha_2+3\alpha_3+5\alpha_4+4\alpha_5+3\alpha_6+2\alpha_7+\alpha_8.
\end{align*}
}

Therefore we can suppose that for all root systems under consideration we marked a sequence $\gamma_1,\gamma_2,\dots
\gamma_k$ of positive roots with following properties:

1. $\gamma_1$ is the maximal root of the system.

2. $\gamma_k$ is a simple root.

3. Every root is obtained from the previous one by subtracting some simple root, and  the first $l$ subtracted simple roots are different.

4. All roots in the systems $A_l,D_l$ are either members of the corresponding sequences, or differences of some two roots from these sequences.
For the systems $E_l$ it holds with some exceptions from the separate special list.

Consider in the matrix $X$ the position $(\mu,\nu)$, $\mu,\nu\in \Phi$.

To find an element on this position in the matrix we need to write all sequences of roots $\beta_1, \dots, \beta_p$ with two following properties:

1. $\mu + \beta_1\in \Phi$, $\mu +\beta_1+\beta_2\in \Phi$, \dots, $\mu+\beta_1+\dots+\beta_i\in \Phi$, \dots,
$\mu+\beta_1+\dots+\beta_p=\nu$.

2. In our initial numerated sequence  $\alpha_1,\dots,\alpha_m, -\alpha_1,\dots, -\alpha_m$ roots $\beta_1,\dots,\beta_p$ are staying strictly from right to left.

Finally, in the matrix $X$ on the position $(\mu,\nu)$ there is sum of all products $\pm \beta_1\cdot\beta_2\dots \beta_p$ by all sequences of roots with these two properties, multiplied by $d_\mu=\lambda s_1^{\langle \alpha_1,\mu\rangle}\dots s_l^{\langle \alpha_l,\mu\rangle}$. If $\mu=\nu$, then we must add~$1$ to the sum.

We shall find the obtained numbers $\lambda, s_1,\dots,s_l,t_1,\dots, t_m,u_1,\dots,u_m$ step by step.

At first we consider in the matrix $X$ the position $(-\gamma_1,-\gamma_1)$. We can not add to the root $-\gamma_1$ any negative root so that in the results we get a root again. If in the sequence $\beta_1,\dots, \beta_p$ the first root is positive, then all other roots have to be positive. Therefore, on the place $(-\gamma_1,-\gamma_1)$ there is just $d_{-\gamma_1}$, and we can know it now. Now consider the place $(-\gamma_1,-\gamma_2)$.  By the same reason  the obtained sequence is only  $\alpha_{k_1}=\gamma_1-\gamma_2$, i.\,e., some simple root $\alpha_{k_1}$. Thus, it is $\pm d_{-\gamma_1} t_{k_1}$ on this place. So we also have found $t_{k_1}$.
If we consider the positions  $(-\gamma_2,-\gamma_2)$ and $(-\gamma_2,-\gamma_1)$, we can see that by similar reasons there are $d_{-\gamma_2}(1\pm u_{k_1}t_{k_1}$ and $\pm d_{-\gamma_2} u_{k_1}$, respectively. So we know $d_{-\gamma_2}$ and $u_{k_1}$.

Now we come to the second step.
From above arguments in the matrix $X$ on the place $(-\gamma_2,-\gamma_3)$ there is $d_{-\gamma_2}(\pm t_{k_2}\pm u_{k_1}t_{k_{1,2}})$, where $\alpha_{k_2}=\gamma_2-\gamma_3$,
$\alpha_{k_{1,2}}=\alpha_{k_1}+\alpha_{k_2}$; on the place $(-\gamma_3,-\gamma_2)$ there is $d_{-\gamma_3}(\pm u_{k_2}\pm u_{k_{1,2}}t_{k_{1}})$; on the place $(-\gamma_1,-\gamma_3)$ there is $d_{-\gamma_1}(\pm t_{k_{1,2}}\pm t_{k_{1}}t_{k_{2}})$  (the second summand can absent, if $k_1< k_2$); on the place $(-\gamma_3,-\gamma_1)$ there is $d_{-\gamma_3}(\pm u_{k_{1,2}}\pm u_{k_{2}}u_{k_{1}})$ (the second summand can absent, if $k_2< k_1$), finally, on the place $(-\gamma_3,-\gamma_3)$ there is $d_{-\gamma_3}(1\pm u_{k_{1,2}}t_{k_{1,2}}\pm u_{k_{1}}t_{k_{1}})$. From these five equations with five variables (from radical) we uniquely can define values of variables, after that we suppose that $d_{-\gamma_1}, d_{-\gamma_2}, d_{-\gamma_3}, t_1, t_2, t_{1,2},u_1, u_2,u_{1,2}$ are known.

Suppose now that we know  numbers $t_i, u_j$ for all indices corresponding to the roots of the form $\gamma_p-\gamma_q$, $1\leqslant p,q< s$, all $d_{-\gamma_r}$, $1\leqslant r<s$, and also $s\leqslant l+1$.
Consider the positions $(-\gamma_1,-\gamma_s)$, $(-\gamma_s,-\gamma_1)$, $(-\gamma_2,-\gamma_s)$, $(-\gamma_s,-\gamma_2)$, \dots, $(-\gamma_{s-1},-\gamma_s)$,
$(-\gamma_s,-\gamma_{s-1})$ and $(-\gamma_s,-\gamma_{s})$ in the matrix~$X$. Clear that on every position $(-\gamma_i,-\gamma_s)$, $1\leqslant i<s$, we have the sum of the number $t_p$ (where $p$ is the number of the root $\gamma_i-\gamma_s$, if it is a root) and products of different numbers $t_a,u_b$, where only one number in the product is not known yet, and all others numbers are known, all of them are from the radical, multiplied by $d_{-\gamma_i}$. The same picture  is on the positions $(-\gamma_s,-\gamma_i)$, $1\leqslant i<s$, but there the single element (without multipliers) is not $t_p$, but $u_p$. On the last place there is $d_{-\gamma_s}(1+\Sigma)$, where $\Sigma$ is also a sum of described type. Therefore we have the  number  of  (not uniform) linear equations, greater to~$1$ than the number of roots $\pm (\gamma_i-\gamma_s)$, with the same number of variables, in every equation precisely one variable has an invertible coefficient, other coefficients are from the radical, for different equations such variable are different. Clear that such system of equation has the solution, and it it unique. Consequently, we have made the induction step and now we know the numbers $t_i, u_j$ for all indices corresponding to the roots of the form $\gamma_p-\gamma_q$, $1\leqslant p,q\leqslant s$, and also $d_{-\gamma_s}$.

 After the $l+1$-th step ($s=l+1$) we know all $d_{-\gamma_1},\dots,d_{-\gamma_{l+1}}$, that uniquely define $\lambda, s_1,\dots, s_l$. After that we know all $d_{-\gamma_i}$, $l+1<i\leqslant \gamma_k$.
 On following steps we do not consider the last position of the form $(-\gamma_i,-\gamma_{i})$,  the  number of equations and variables decreases to one.

On the last step we know numbers $t_i, u_j$ for all indices corresponding to the roots of the form $\gamma_p-\gamma_q$, $1\leqslant p,q
\leqslant k$. Consider now in the matrix $X$ the positions $(-\gamma_1,h_{\gamma_1})$, $(h_{\gamma_1},-\gamma_1)$, $(-\gamma_2,h_{\gamma_2})$, $(h_{\gamma_2},-\gamma_2)$, \dots, $(-\gamma_k,h_{\gamma_k})$, $(h_{\gamma_k},-\gamma_k)$. Completely similar to the previous arguments we can find all coefficients  $t$ and $u$, corresponding to the roots $\pm \gamma_1,\dots, \pm \gamma_k$.

After that in the systems $A_l$ and $D_l$ we know all coefficients in the product;  moreover, in the proof of the lemma we have described the correspondence between them and positions of the matrix~$X$.

In the systems $E_l$ there are some roots with unknown coefficients. Let us show what to do with it on the system~$E_8$.
Order the excepted roots by its height and call them $\beta_1,\dots, \beta_{11}$. Note that $\delta_1:=\gamma_2-\beta_1\in \Phi, \gamma_1-\beta_1\notin \Phi$,
$\delta_2:=\gamma_2-\beta_2\in \Phi, \gamma_1-\beta_2\notin \Phi$, $\delta_3:=\gamma_2-\beta_3\in \Phi, \gamma_1-\beta_3\notin \Phi$, $\delta_4:=\gamma_1-\beta_4\in \Phi$,\dots, $\delta_7:=\gamma_1-\beta_7\in \Phi$, $\delta_8:=\gamma_2-\beta_8\in \Phi, \gamma_1-\beta_8\notin \Phi$, $\delta_9:=\gamma_1-\beta_9\in \Phi$, $\delta_{10}:=\gamma_1-\beta_{10}\in \Phi$, $\delta_{11}:=\gamma_2-\beta_{11}\in \Phi, \gamma_1-\beta_{11}\notin \Phi$.

Let us start with the root $\beta_1$. Consider in the matrix $X$ the position $(-\gamma_2,-\delta_1)$. On this position in the matrix there is a sum $t_p$, corresponding to the root~$\beta_1$, and products of numbers $t_i,u_j$, corresponding to the roots of smaller height. Since for all heights smaller than the height of $\beta_1$ we already know coefficients $t,u$, we can find the obtained coefficient  $t_p$ directly. Similarly we find a coefficient~$u_p$, considering the position $(-\delta_1,-\gamma_2)$.

Now all coefficients for all roots of heights smaller than the height of $\beta_2$, are found. According to it we can completely repeat the previous arguments with the positions  $(-\gamma_2,-\delta_2)$ and $(-\delta_2,-\gamma_2)$. Then we act similarly but for the roots $\beta_4,\beta_5,\beta_6,\beta_7.\beta_9,\beta_{10}$ in the matrix $X$ we consider not positions $(-\gamma_2,-\delta_i)$ and $(-\delta_i,-\gamma_2)$, but
$(-\gamma_1,-\delta_i)$ and $(-\delta_i,-\gamma_1)$.

Therefore lemma is completely proved.
\end{proof}

Now we return to our main proof.
Recall that we work with the matrix~$C$, equivalent to the unit matrix modulo radical, and mapping elementary Chevalley group to itself.

For every root $\alpha\in\Phi$ we have the equation
\begin{equation}\label{osn_eq}
C x_{\alpha}(1)C^{-1}=x_{\alpha}(1)\cdot g_\alpha,\quad g_\alpha\in
G_{\ad}(\Phi,R,J).
\end{equation}
Every element $g_\alpha \in G_{\ad}(\Phi,R,J)$ can be decomposed in the product
\begin{equation}\label{razl_rad}
 t_{\alpha_1}(1+a_1)\dots
t_{\alpha_l}(1+a_l)x_{\alpha_1}(b_1)\dots
x_{\alpha_m}(b_m)x_{\alpha_{-1}}(c_1)\dots x_{\alpha_{-m}}(c_m),
\end{equation}
where $a_1,\dots,a_l,b_1,\dots,b_m,c_1,\dots, c_m\in J$ (see,
for example,~\cite{Abe1}).

Let $C=E+X=E+(x_{i,j})$. Then for every root~$\alpha\in \Phi$
we can write a matrix equation~\ref{osn_eq} with variables $x_{i,j},
a_1,\dots,a_l,b_1,\dots,b_m,c_1,\dots, c_m$, every of them is from the radical.

We change these equations.
We consider the matrix~$C$ and ``imagine'', that it is a matrix from Lemma~\ref{prod2}. Then by some its concrete $n+1$~positions we can  ``reproduce'' all coefficients $\lambda, s_1,\dots,s_l, t_1,\dots,t_m,u_1,\dots, u_m$ in decomposition of matrix in the product from Lemma~\ref{prod2}. In the result we obtain some matrix $D$ from
elementary Chevalley group, every its coefficient is some (known) function of coefficients of the matrix~$C$.
Change now the equations~\eqref{osn_eq} to the equations
\begin{equation}\label{fol_eq}
D^{-1}C x_{\alpha}(1)C^{-1}D=x_{\alpha}(1)\cdot {g_\alpha}',\quad {g_\alpha}'\in
G_{\ad}(\Phi,R,J).
\end{equation}
We again get matrix equations but with variables $y_{i,j},
a_1',\dots,a_l',b_1',\dots,b_m',c_1',\dots, c_m'$, every of them also is from the radical, and every $y_{p,q}$ is some known function from  $x_{i,j}$. We denote the matrix $D^{-1}C$ by~$C'$.

We need to show that a solution exists only if all variables with  primes are zeros. Some $x_{i,j}$ also have to be zeros, and others disappear in equations. Since the equations are too complicated we shall consider the linearized system of equations. It is sufficient to show that all variables that do not disappear in the linearized system (suppose that there are  $q$~such variables) are members of some subsystem consisting from $q$ inear equations with an invertible (in the ring~$R$) determinant.

in other words, from the matrix equations we shall show step by step that all variables are zeros.

Clear that linearization of the product $Y^{-1}(E+X)$ gives some matrix $E+(z_{i,j})$, that have zeros on all positions described in Lemma~\ref{prod2}.

To find the final form of the linearized system we write it as follows:
\begin{multline*}
(E+Z)x_\alpha(1) =x_\alpha(1)(E+a_1T_1+a_1^2\dots)\dots
(E+a_lT_l+a_1^2\dots)\cdot\\
\cdot(E+b_1X_{\alpha_1}+b_1^2X_{\alpha_1}^2/2)\dots
(E+c_mX_{-\alpha_1}+c_m^2X_{-\alpha_m}^2/2)(E+Z),
\end{multline*}
where $X_\alpha$ the corresponding Lie algebra element in the adjoint representation, the matrix $T_i$ is diagonal, has on the diagonal~$p$ at the place corresponding to the vector $v_k$ iff in decomposition of $\alpha_k$ in the sum of simple roots the root
 $\alpha_i$ entries in this decomposition  $p$~times ($p$
can be zero or negative); on the places corresponding to the basis vectors  $V_j$, the matrix has zeros on the diagonal.

Then the linearized system is
$$
Zx_{\alpha}(1)-x_{\alpha}(1)(Z+a_1T_1+\dots+a_lT_l+b_1X_{\alpha_1}+\dots+c_mX_{\alpha_m})=0.
$$
Clear that for simple roots $\alpha_i$ the summand $b_i X_{\alpha_i}$ is absent; besides, $2m$ fixed elements of~$Z$ are zeros.
This equation can be written for every $\alpha\in
\Phi$ (naturally, with other  $a_j, b_j, c_j$), and can be written only for the roots generating the Chevalley groups, i.\,e., for $\alpha_1,\dots,
\alpha_l, -\alpha_1, \dots, -\alpha_l$. The number free variables is not changed.

In the next section we prove the obtained equation for root systems $A_2$, and latter we shall give the prove for all systems under consideration.

\section{Linear systems in the case $A_2$.}

In this case we have four conditions:
$$
\begin{cases}
Zx_{\alpha_1}(1)-x_{\alpha_1}(1)(X+a_{1,1}T_1+a_{2,1}T_2+\\
\ \ \ \ \
+b_{1,1}X_{\alpha_1}+b_{2,1}X_{\alpha_2}+b_{3,1}X_{\alpha_1+\alpha_2}
+c_{1,1}X_{-\alpha_1}+c_{2,1}X_{-\alpha_2}+c_{3,1}X_{-\alpha_1-\alpha_2})=0;\\
Zx_{\alpha_2}(1)-x_{\alpha_2}(1)(X+a_{1,2}T_1+a_{2,2}T_2+\\
\ \ \ \ \ +b_{1,2}X_{\alpha_1}+b_{2,2}X_{\alpha_2}+b_{3,2}X_{\alpha_1+\alpha_2}
+c_{1,2}X_{-\alpha_1}+c_{2,2}X_{-\alpha_2}+c_{3,2}X_{-\alpha_1-\alpha_2})=0;\\
Xx_{-\alpha_1}(1)-x_{-\alpha_1}(1)(X+a_{1,3}T_1+a_{2,3}T_2+\\
\ \ \ \ \
+b_{1,3}X_{\alpha_1}+b_{2,3}X_{\alpha_2}+b_{3,3}X_{\alpha_1+\alpha_2}
+c_{1,3}X_{-\alpha_1}+c_{2,3}X_{-\alpha_2}+c_{3,3}X_{-\alpha_1-\alpha_2})=0;\\
Xx_{-\alpha_2}(1)-x_{-\alpha_2}(1)(X+a_{1,4}T_1+a_{2,4}T_2+\\
\ \ \ \ \
+b_{1,4}X_{\alpha_1}+b_{2,4}X_{\alpha_2}+b_{3,4}X_{\alpha_1+\alpha_2}
+c_{1,4}X_{-\alpha_1}+c_{2,4}X_{-\alpha_2}+c_{3,4}X_{-\alpha_1-\alpha_2})=0.
\end{cases}
$$
Also
$$
\begin{cases}
T_1= e_{1,1}-e_{2,2}+e_{5,5}-e_{6,6};\\
T_2= e_{3,3}-e_{4,4}+e_{5,5}-e_{6,6};\\
X_{\alpha_1}= -2e_{1,7}+e_{1,8}+e_{4,6}-e_{3,5}+e_{7,2};\\
X_{\alpha_2}=e_{3,7}-2e_{3,8}+e_{2,6}-e_{5,1}+e_{8,4};\\
X_{\alpha_1+\alpha_2}=
-e_{5,7}-e_{5,8}-e_{1,4}+e_{3,2}+e_{7,6}+e_{8,6};\\
X_{-\alpha_1}=-2e_{2,7}+e_{2,8}+e_{3,5}-e_{6,4}+e_{7,1};\\
X_{-\alpha_2}=e_{4,7}-2e_{4,8}-e_{1,5}+e_{6,2}+e_{8,3};\\
X_{-\alpha_1-\alpha_2}=-e_{6,7}-e_{6,8}-e_{2,3}+e_{4,1}+e_{7,5}+e_{8,5}.
\end{cases}
$$

The matrix $Z$ is
$$
\begin{pmatrix}
z_{11}& z_{12} & z_{13} & z_{14} & z_{15} & z_{16} & z_{17} & z_{18}\\
z_{21}& 0 & z_{23} & z_{24} & z_{25} & 0& z_{27} & z_{28}\\
z_{31}& z_{32} & z_{33} & z_{34} & z_{35} & z_{36} & z_{37} & z_{38}\\
z_{41}& z_{42} & z_{43} & 0 & z_{45} & 0& z_{47} & z_{48}\\
z_{51}& z_{52} & z_{53} & z_{54} & z_{55} & z_{56} & z_{57} & z_{58}\\
z_{61}& 0 & z_{63} & 0 & z_{65} & 0& z_{67} & 0\\
z_{71}& z_{72} & z_{73} & z_{74} & z_{75} & z_{76} & z_{77} & z_{78}\\
z_{81}& z_{82} & z_{83} & z_{84} & z_{85} & 0& z_{87} & z_{88}
\end{pmatrix}.
$$

Directly writing all equations of the system and solving them by steps, we obtain that all variables  $a_i,b_i,c_i$ are zero, and the matrix $Z$ is scalar.

Therefore for the system  $A_2$ theorem is proved.

\section{Linear systems in other cases.}

Suppose now that we deal with  arbitrary root system (from the types under consideration) of rank $>2$. We  consider $2l$ conditions in this situation (just for simple roots and opposite to them):
$$
\begin{cases}
Zx_{\alpha_1}(1)-x_{\alpha_1}(1)(Z+a_{1,1}T_1+\dots +a_{l,1}T_l+\\
\ \ \ \ \
+b_{1,1}X_{\alpha_1}+\dots+b_{m,1}X_{\alpha_m}
+c_{1,1}X_{-\alpha_1}+\dots+c_{m,1}X_{-\alpha_m})=0;\\
\dots\\
Zx_{\alpha_l}(1)-x_{\alpha_l}(1)(Z+a_{1,l}T_1+\dots +a_{l,l}T_l+\\
\ \ \ \ \
+b_{1,l}X_{\alpha_1}+\dots+X_{\alpha_m}b_{m,1}X_{\alpha_m}
+c_{1,l}X_{-\alpha_1}+\dots+c_{m,l}X_{-\alpha_m})=0;\\
Zx_{-\alpha_1}(1)-x_{-\alpha_1}(1)(Z+a_{1,l+1}T_1+\dots+a_{l,l+1}T_l+\\
\ \ \ \ \
+b_{1,l+1}X_{\alpha_1}+\dots+b_{m,l+1}X_{\alpha_m}
+c_{1,l+1}X_{-\alpha_1}+\dots+c_{m,l+1}X_{-\alpha_m})=0;\\
\dots\\
Zx_{-\alpha_l}(1)-x_{-\alpha_l}(1)(Z+a_{1,2l}T_1+\dots+a_{l,2l}T_l+\\
\ \ \ \ \
+b_{1,2l}X_{\alpha_1}+\dots+b_{m,2l}X_{\alpha_m}
+c_{1,2l}X_{-\alpha_1}+\dots+c_{m,2l}X_{-\alpha_m})=0.
\end{cases}
$$

In the case of root systems under consideration
$$
T_i=e_{\alpha_i,\alpha_i}-e_{-\alpha_i,-\alpha_i}+\sum_{j\ne i}
k_j(e_{\alpha_j,\alpha_j}-e_{-\alpha_j,-\alpha_j}),
$$
where $\alpha_j=k_j\alpha_i+\alpha$, the simple root $\alpha_i$ does not entry in decomposition of~$\alpha$ in the sum of simple roots;
$$
X_{\alpha_p}=e_{h_p,-\alpha_p}-\sum_{q=1}^l \langle
\alpha_p,\alpha_q\rangle e_{\alpha_p,h_q} +\sum_{s,t:
\alpha_p+\alpha_s=\alpha_t}
e_{-\alpha_s,-\alpha_t}-e_{\alpha_t,\alpha_s}
$$
for all $\alpha_p\in \Phi$.

Suppose that we fixed the obtained linear uniform system of equations with described $n+1$ zero positions. Recall that our aim is to show that all values  $z_{i,j}$, $a_{s,t}, b_{s,t}, c_{s,t}$ are zeros.

At first we consider the pair of equations with numbers~$1$ and~$l+1$. Clear that all other corresponding pairs have the same construction
 (since there is a substitution of basis that moves any root to any other root).

We can rename basis so that the matrices $x_{\alpha_1}(1)$ and
$x_{-\alpha_1}(1)$ have the forms of three diagonal blocks: the first one is corresponded to the basis vextors $\{ x_{\alpha_1}, x_{-\alpha_1},
h_1, h_2\}$, the second one corresponded to $\{ x_\alpha\mid \langle \alpha_1,
\alpha\rangle=\pm 1\}$, and the third corresponded to $\{ x_\alpha\mid \langle
\alpha_1,\alpha\rangle=0; h_3,\dots, h_l\}$. Also the third block is just a unit matrix. Let

$$
x_{\alpha_1}(1)=\begin{pmatrix} A& 0& 0\\
0& B& 0\\
0& 0& E
\end{pmatrix},\quad
x_{-\alpha_1}(1)=\begin{pmatrix} A'& 0& 0\\
0& B'& 0\\
0& 0& E
\end{pmatrix}
$$
and suppose that all matrices $Y$ under consideration also have the similar block form:
$$
Y=\begin{pmatrix} Y_{11}& Y_{12}& Y_{13}\\
Y_{21}& Y_{22}& Y_{23}\\
Y_{31}& Y_{32}& Y_{33}
\end{pmatrix}.
$$

Denote the matrix
$a_{1,1}T_1+a_{2,1}T_2+\dots+a_{l,1}T_l+b_{1,1}X_{\alpha_1}+
\dots+b_{m,1}X_{\alpha_m}+c_{1,1}X_{-\alpha_1}+\dots+c_{m,1}X_{-\alpha_m}$
by $D$. Note that from the basis construction  $D_{1,3}=D_{3,1}=0$.
Then the first equation is
$$
\begin{pmatrix}
Z_{11}A& Z_{12}B& Z_{13}\\
Z_{21}A& Z_{22}B& Z_{23}\\
Z_{31}A& Z_{32}B& Z_{33}
\end{pmatrix}=
\begin{pmatrix}
AZ_{11}& AZ_{12}& AZ_{13}\\
BZ_{21}& BZ_{22}& BZ_{23}\\
Z_{31}& Z_{32}& Z_{33}
\end{pmatrix}+
\begin{pmatrix}
AD_{11}& AD_{12}& 0\\
BD_{21}& BD_{22}& BD_{23}\\
0 & D_{32}& D_{33}
\end{pmatrix}.
$$
From the written equality directly follows that $D_{33}=0$.
Therefore,  $a_{3,1}=\dots=a_{l,1}=0$, $b_{k,1}=c_{k,1}=0$ for
$\langle \alpha_1, \alpha_k\rangle=0$.

 Now consider a positive  root $\alpha$ such that $\beta=\alpha+\alpha_1\in \Phi$ and  $\alpha,\beta$ are from the first $l+1$ members of the sequence $\gamma_1,\dots, \gamma_k$.
 For the root system $A_l$ it is  $\gamma_2$, for $D_l$ it is $\gamma_3$, for $E_8$ it is $\gamma_3$ (similar pairs of roots in the sequence $\gamma_1,\dots, \gamma_k$ can be found also for $\alpha_2,\dots, \alpha_l$ according to the porperty~3 of the sequence $\gamma_1,\dots, \gamma_k$).

Consider the basis part $\alpha,-\alpha,\beta,-\beta$. For the matrices $x_{\alpha_1}(1)$ and
$x_{-\alpha_1}(1)$ this part of basis is a direct summand, so that we can study it independently. By the construction
 we know that $z_{-\alpha,-\alpha}=z_{-\beta,-\beta}=z_{-\alpha,-\beta}=z_{\beta,-\alpha}=0$. So we have the condition
\begin{multline*}
\begin{pmatrix}
z_{\alpha,\alpha}& z_{\alpha,-\alpha}& z_{\alpha,\beta}& z_{\alpha,-\beta}\\
z_{-\alpha,\alpha}& 0& z_{-\alpha,\beta}& 0\\
z_{\beta,\alpha}& z_{\beta,-\alpha}& z_{\beta,\beta}& z_{\beta,-\beta}\\
z_{-\beta,\alpha}& 0& z_{-\beta,\beta}& 0
\end{pmatrix}
\begin{pmatrix}
1& 0& 0& 0\\
0& 1& 0& -1\\
1& 0& 1& 0\\
0& 0& 0& 1
\end{pmatrix}=\\ =
\begin{pmatrix}
1& 0& 0& 0\\
0& 1& 0& -1\\
1& 0& 1& 0\\
0& 0& 0& 1
\end{pmatrix}
\left(\begin{pmatrix}
z_{\alpha,\alpha}& z_{\alpha,-\alpha}& z_{\alpha,\beta}& z_{\alpha,-\beta}\\
z_{-\alpha,\alpha}& 0& z_{-\alpha,\beta}& 0\\
z_{\beta,\alpha}& z_{\beta,-\alpha}& z_{\beta,\beta}& z_{\beta,-\beta}\\
z_{-\beta,\alpha}& 0& z_{-\beta,\beta}& 0
\end{pmatrix}+
\begin{pmatrix}
a_{\alpha} & 0& -c_{1,1}& 0\\
0& -a_{\alpha}& 0& -b_{1,1}\\
b_{1,1}& 0& a_{\beta}& 0& 0\\
0& c_{1,1}& 0& -a_{\beta}
\end{pmatrix}\right) .
\end{multline*}

As the result we obtain that the following matrix must be zero:
$$
\begin{pmatrix}
z_{\alpha,\beta}-a_{\alpha}& 0& c_{1,1}& z_{\alpha,-\alpha}\\
z_{-\alpha,\beta}+z_{-\beta,\alpha}& c_{1,1}+a_{\alpha}& z_{-\beta,\beta}& b_{1,1}-a_{\beta}\\
z_{\beta,\beta}-z_{\alpha,\alpha}-a_{\alpha}-b_{1,1}& -z_{\alpha,-\alpha}& -z_{\alpha,\beta}+c_{1,1}-a_{\beta}&
-z_{\beta,-\alpha}-z_{\alpha,-\beta}\\
z_{-\beta,\beta}& -c_{1,1}& 0& a_{\beta}
\end{pmatrix}.
$$
Consequently, $a_{\alpha}=a_{\beta}=b_{1,1}=c_{1,1}=0$. Since we already know that $a_{3,1}=\dots=a_{l,1}=0$,
we have $a_{1,1}=a_{2,1}=0$.

Consider such a root $\alpha\in \Phi^+$, that $\alpha+\alpha_1\in \Phi^+$. In the sequence $\gamma_1,\dots, \gamma_k$ we find such roots $\gamma_p$ and $\gamma_q$, that $\gamma_q-\gamma_p=\alpha$. Since all roots in our systems have the same length, the reflection coefficients $\langle\ ,\ \rangle$ are bilinear up to the roots. We know that $\langle \alpha,\alpha_1\rangle =-1$, and for the numbers  $\langle \gamma_p,\alpha_1\rangle$ and $\langle \gamma_q,\alpha_1\rangle$ are just three possibilities: $0,1,-1$. So we can have only two situations, considered below:

1) the root $\gamma_p$ is orthogonal to the root~$\alpha_1$, and the root $\gamma_q$ is not orthogonal. Then $\gamma'=\gamma_q+\alpha_1$ is also a root, we can consider the basis part $\gamma_p,-\gamma_p,\gamma_q,-\gamma_q, \gamma',-\gamma'$, that is an invariant direct summand for $x_{\alpha_1}(1)$. On this basis part
$$
x_{\alpha_1}(1)=\begin{pmatrix}
1& 0& 0& 0& 0& 0\\
0& 1& 0& 0& 0& 0\\
0& 0& 1& 0& 0& 0\\
0& 0& 0& 1& 0& -1\\
0& 0& 1& 0& 1& 0\\
0& 0& 0& 0& 0& 1
\end{pmatrix},
$$
the matrix $Z$ has the positions   $z_{-\gamma_p,-\gamma_q}$ and $z_{-\gamma_q,-\gamma_p}$ equal to zero (by Lemma~\ref{prod2}), the matrix $D$ is
$$
\begin{pmatrix}
0& 0& -c_{\alpha,1}& 0& -c_{\alpha+\alpha_1,1}& 0\\
0& 0& 0& -b_{\alpha,1}& 0& -b_{\alpha+\alpha_1}\\
b_{\alpha,1}& 0& 0& 0& 0& 0\\
0& c_{\alpha,1}& 0& 0& 0& 0\\
b_{\alpha+\alpha_1,1}& 0& 0& 0& 0& 0\\
0& c_{\alpha+\alpha_1,1}& 0& 0& 0& 0
\end{pmatrix}.
$$

Now from the main condition the following matrix is zero:
{\tiny
$$
\begin{pmatrix}
0& 0& z_{\gamma_p,\gamma'}+c_{\alpha,1}& 0& c_{\alpha+\alpha_1,1}& -z_{\gamma_p,-\gamma_q}\\
0& 0& z_{-\gamma_p,\gamma'}& b_{\alpha,1}& 0& b_{\alpha+\alpha_1,1}\\
-b_{\alpha,1}& 0& z_{\gamma_q,\gamma'}& 0& 0& -z_{\gamma_q,-\gamma_q}\\
z_{-\gamma',\gamma_p}& z_{-\gamma',-\gamma_p}-c_{\alpha,1}+c_{\alpha+\alpha_1,1}& z_{-\gamma_q,\gamma'}+z_{-\gamma',\gamma_q}& z_{-\gamma',-\gamma_q}& z_{-\gamma',\gamma'}&
z_{-\gamma',-\gamma'}-z_{-\gamma_q,-\gamma_q}\\
-z_{\gamma_q,\gamma_p}-b_{\alpha,1}-b_{\alpha+\alpha_1,1}& -z_{\gamma_q,-\gamma_p}& z_{\gamma',\gamma'}-z_{\gamma_q,\gamma_q}& -z_{\gamma_q,-\gamma_q}& -z_{\gamma_q,\gamma'}&
-z_{\gamma',-\gamma_q}-z_{\gamma_q,-\gamma'}\\
0& -c_{\alpha+\alpha_1}& z_{-\gamma',\gamma'}& 0& 0& z_{-\gamma',-\gamma_q}
\end{pmatrix}.
$$
}

Clear that $b_{\alpha,1}=b_{\alpha+\alpha_1,1}=c_{\alpha+\alpha_1,1}=0$.

2) In the second case we have $\gamma''=\gamma_p-\alpha_1\in \Phi^+$, $\langle \gamma_q,\alpha_1\rangle =0$. This case is similar to the previous one, and on the basis part $\gamma_p,-\gamma_p,\gamma_q,-\gamma_q,\gamma'',-\gamma''$
the following matrix is zero:
{\tiny
$$
\begin{pmatrix}
-z_{\gamma'',\gamma_p}& -z_{\gamma_p,-\gamma''}-z_{\gamma'',-\gamma_p}& -z_{\gamma'',\gamma_q}+c_{\alpha,1}+c_{\alpha+\alpha_1,1}& -z_{\gamma'',-\gamma_q}& z_{\gamma_p,\gamma_p}-z_{\gamma'',\gamma''}& -z_{\gamma'',-\gamma''}\\
 0& -z_{-\gamma_p,-\gamma''}& 0& b_{\alpha,1}& z_{-\gamma_p,\gamma_p}& 0\\
 -b_{\alpha,1}& -z_{\gamma_q,-\gamma''}& 0& 0& z_{\gamma_q,\gamma_p}-b_{\alpha+\alpha_1,1}& 0\\
 0& -z_{-\gamma+q,-\gamma''}-c_{\alpha,1}& 0& 0& z_{-\gamma_q,\gamma_p}& -c_{\alpha+\alpha_1,1}\\
 0& -z_{\gamma'',-\gamma''}&  c_{\alpha+\alpha_1,1}& 0& z_{\gamma'',\gamma_p}& 0\\
 z_{-\gamma_p,\gamma_p}& -z_{-\gamma'',-\gamma''}+z_{-\gamma_p.-\gamma_p}& z_{-\gamma_p,\gamma_q}&
  b_{\alpha+\alpha_1,1}-b_{\alpha,1}& z_{-\gamma'',\gamma_p}+z_{-\gamma_p,\gamma''}& z_{-\gamma_p,-\gamma''}
  \end{pmatrix}.
  $$
}

We again have $b_{\alpha,1}=b_{\alpha+\alpha_1,1}=c_{\alpha+\alpha_1,1}=0$.

Considering a pair of roots $\gamma_s,\gamma_t$ with the property $\gamma_t-\gamma_s=\alpha+\alpha_1$ we similarly obtain two cases and $c_{\alpha,1}=0$.

Therefore we proved that $a_{i,1}=0$ for every $i=1,\dots,l$, $b_{\alpha,1}=c_{\alpha,1}=0$ for every $\alpha\in \Phi$. Consequently, the matrix~$Z$ commutes with $x_{\alpha_1}(1)$.

Similarly from other conditions we obtain that the matrix $Z$ commutes with all $x_{\alpha}(1)$, $\alpha\in \Phi$, and so it is scalar. Since $Z$ has zeros on the diagonal (by the construction) it is zero.
 Theorem~\ref{norm} is proved.

\section{Proof of the main theorem.}

\begin{lemma}\label{tor}
If $\overline t\in G_{\ad}(\Phi,R)$ is a torus element $(\Phi$ is one of the systems under consideration, $R$ is a local commutative ring\emph{)}, then there exists such a torus element $t\in G_\pi(\Phi,S)$ that a ring $S$ contains~$R$, $t$ lies in the normalizer of $G_\pi(\Phi,R)$, under factorization of $G_\pi(\Phi,S)$ by its center $t$ gives~$\overline t$.
\end{lemma}

\begin{proof}
Clear that it is sufficient to prove the lemma statement for basic elements of the torus $T_{\ad}(\Phi,R)$. Since all roots in the root systems under consideration are conjugate, we can take only one torus element: $\overline t=\chi_{\alpha_1}(r)$. This element acts on the elementary subgroup as follows:
$\overline t x_{\alpha}(s) \overline t^{-1}=x_{\alpha}(r^ks)$, where $\alpha=k\alpha_1+\beta$, $\alpha_1$ does not enter to the decomposition of $\beta$ into the sum of simple roots.

Clear that it is sufficient to construct such an extension  $S$ of~$R$  and a torus element $t$ of $G_\pi(\Phi,S)$ that $t x_{\alpha}(s) t^{-1}= \overline t x_\alpha(s) \overline t^{-1}$ for all $\alpha\in \Phi$, $s\in R$.

Clear that if the constructed torus element acts rightly on simple roots, it automatically acts rightly on all roots.

Consider three types of root systems under consideration separately.

{\bf Root systems  $A_l$.}
Consider a ring~$S$, that obtained from~$R$ by adding a root of the $l+1$-th power from~$r$ (namely, $S=R[x]/(x^{l+1}-r)$). Denote this root by~$s$. Then the obtained element is
$$
t=h_{\alpha_1}(s^l)h_{\alpha_2}(s^{l-1})\dots h_{\alpha_{l-1}}(s^2)h_{\alpha_l}(s).
$$

Namely,
\begin{multline*}
t x_{\alpha_1}(u) t^{-1}=h_{\alpha_1}(s^l) h_{\alpha_2}(s^{l-1}) x_{\alpha_1}(u) h_{\alpha_2}(s^{l-1})^{-1} h_{\alpha_1}(s^l)^{-1}=\\
=h_{\alpha_1}(s^l) x_{\alpha_1}(u/ s^{l-1}) h_{\alpha_1}(s^l)^{-1}= x_{\alpha_1}(u\cdot s^{2l}/s^{l-1})=\\
=x_{\alpha_1}(us^{l+1})=x_{\alpha_1}(u r),
\end{multline*}
for every simple root $\alpha_i$, $1<i< l$,
\begin{multline*}
t x_{\alpha_i}(u) t^{-1}=h_{\alpha_{i-1}}(s^{l-i+1}) h_{\alpha_i}(s^{l-i}) h_{\alpha_{i+1}}(s^{l-i-1})  x_{\alpha_i}(u) h_{\alpha_{i+1}}(s^{l-i-1})^{-1} h_{\alpha_i}(s^{l-i})^{-1} h_{\alpha_{i-1}}(s^{l-i+1})^{-1}=\\
 =x_{\alpha_i}(u\cdot s^{2l-2i}/(s^{l-i-1}s^{l-i+1})) =x_{\alpha_1}(u),
\end{multline*}
finally, for the root $\alpha_l$
$$
t x_{\alpha_l}(u) t^{-1}=h_{\alpha_{l-1}}(s^2) h_{\alpha_l}(s)  x_{\alpha_i}(u)  h_{\alpha_l}(s^2)^{-1} h_{\alpha_{l-1}}(s)^{-1}=x_{\alpha_i}(u\cdot s^2/s^2) =x_{\alpha_1}(u).
$$

\bigskip

{\bf Root systems  $D_l$.}
In the given case  $S$ is obtained from~$R$ by adding of square root from~$r$ (again denote it by~$s$). Here we can take the element
$$
t=h_{\alpha_1}(r)h_{\alpha_2}(r)\dots h_{\alpha_{l-2}}(r)h_{\alpha_{l-1}}(s)h_{\alpha_l}(s).
$$
Let us check it.
At first,
\begin{multline*}
t x_{\alpha_1}(u) t^{-1}=h_{\alpha_1}(r) h_{\alpha_2}(r) x_{\alpha_1}(u) h_{\alpha_2}(r)^{-1} h_{\alpha_1}(r)^{-1}=\\
=h_{\alpha_1}(r) x_{\alpha_1}(u/ r) h_{\alpha_1}(r)^{-1}= x_{\alpha_1}(u\cdot r^{2}/r)
=x_{\alpha_1}(u r),
\end{multline*}
for every simple root $\alpha_i$, $1<i< l-2$,
\begin{multline*}
t x_{\alpha_i}(u) t^{-1}=h_{\alpha_{i-1}}(r) h_{\alpha_i}(r) h_{\alpha_{i+1}}(r)  x_{\alpha_i}(u) h_{\alpha_{i+1}}(r)^{-1} h_{\alpha_i}(r)^{-1} h_{\alpha_{i-1}}(r)^{-1}=\\
 =x_{\alpha_i}(u\cdot r^2/(r\cdot r)) =x_{\alpha_i}(u),
\end{multline*}
for the root $\alpha_{l-2}$
\begin{multline*}
t x_{\alpha_{l-2}}(u) t^{-1}=h_{\alpha_{l-3}}(r) h_{\alpha_{l-2}}(r) h_{\alpha_{l-1}}(s) h_{\alpha_{l}}(s)  x_{\alpha_{l-2}}(u) h_{\alpha_{l-3}}(r)^{-1} h_{\alpha_{l-2}}(r)^{-1} h_{\alpha_{l-1}}(s)^{-1} h_{\alpha_{l}}(s)^{-1}=\\
 =x_{\alpha_{l-2}}(u\cdot r^2/(r\cdot s\cdot s)) =x_{\alpha_{l-2}}(u),
\end{multline*}
finally, for the root $\alpha_l$ (similarly for $\alpha_{l-1}$)
$$
t x_{\alpha_l}(u) t^{-1}=h_{\alpha_{l-2}}(r) h_{\alpha_l}(s)  x_{\alpha_l}(u)  h_{\alpha_l}(s)^{-1} h_{\alpha_{l-2}}(r)^{-1}
 =x_{\alpha_l}(u\cdot s^2/r) =x_{\alpha_l}(u).
$$

\bigskip

{\bf Root systems $E_6,E_7,E_8$.}
For the root system $E_6$ we take $s=\sqrt[3]{r}$ and
$$
t=h_{\alpha_1}(s^4)h_{\alpha_2}(s^3)h_{\alpha_3}(s^5)h_{\alpha_4}(s^6)h_{\alpha_5}(s^4)h_{\alpha_6}(s^2).
$$
Actually,
\begin{multline*}
t x_{\alpha_1}(u) t^{-1}=h_{\alpha_1}(s^4) h_{\alpha_3}(s^5) x_{\alpha_1}(u) h_{\alpha_3}(s^5)^{-1} h_{\alpha_1}(s^4)^{-1}=\\
=h_{\alpha_1}(s^4) x_{\alpha_1}(u/ s^5) h_{\alpha_1}(s^4)^{-1}= x_{\alpha_1}(u\cdot s^8/s^5)
=x_{\alpha_1}(u s^3)=x_{\alpha_1}(u\cdot r),
\end{multline*}
for $\alpha_2$
$$
t x_{\alpha_2}(u) t^{-1}=h_{\alpha_2}(s^3) h_{\alpha_4}(s^6) x_{\alpha_2}(u) h_{\alpha_4}(s^6)^{-1} h_{\alpha_2}(s^3)^{-1}
=x_{\alpha_2}(u\cdot (s^3)^2/s^6)=x_{\alpha_2}(u ),
$$
for $\alpha_3$
\begin{multline*}
t x_{\alpha_3}(u) t^{-1}=h_{\alpha_1}(s^4) h_{\alpha_3}(s^5) h_{\alpha_4}(s^6) x_{\alpha_3}(u) h_{\alpha_4}(s^6)^{-1} h_{\alpha_3}(s_5)^{-1} h_{\alpha_1}(s^4)^{-1}=\\
=x_{\alpha_3}(u\cdot (s^5)^2/(s^4\cdot s^6))=x_{\alpha_3}(u ),
\end{multline*}
for $\alpha_4$
\begin{multline*}
t x_{\alpha_4}(u) t^{-1}=h_{\alpha_2}(s^3) h_{\alpha_3}(s^5) h_{\alpha_4}(s^6) h_{\alpha_5}(s^4) x_{\alpha_4}(u) h_{\alpha_5}(s^4) h_{\alpha_4}(s^6)^{-1} h_{\alpha_3}(s_5)^{-1} h_{\alpha_2}(s^3)^{-1}=\\
=x_{\alpha_4}(u\cdot (s^6)^2/(s^4 \cdot s^5\cdot s^3))=x_{\alpha_4}(u ),
\end{multline*}
for $\alpha_5$
\begin{multline*}
t x_{\alpha_5}(u) t^{-1}=h_{\alpha_4}(s^6) h_{\alpha_5}(s^4) h_{\alpha_6}(s^2) x_{\alpha_5}(u) h_{\alpha_6}(s^2)^{-1} h_{\alpha_5}(s_4)^{-1} h_{\alpha_4}(s^6)^{-1}=\\
=x_{\alpha_5}(u\cdot (s^4)^2/(s^6\cdot s^2))=x_{\alpha_5}(u ),
\end{multline*}
finally, for $\alpha_6$
$$
t x_{\alpha_6}(u) t^{-1}=h_{\alpha_5}(s^4) h_{\alpha_6}(s^2) x_{\alpha_6}(u) h_{\alpha_6}(s^2)^{-1} h_{\alpha_5}(s^4)^{-1}
=x_{\alpha_6}(u\cdot (s^2)^2/s^4)=x_{\alpha_6}(u ).
$$

\bigskip

for the root system $E_7$ we have $S=R$,
$$
t=h_{\alpha_1}(r^2)h_{\alpha_2}(r^2)h_{\alpha_3}(r^3)h_{\alpha_4}(r^4)h_{\alpha_5}(r^3)h_{\alpha_6}(r^2)h_{\alpha_7}(r).
$$
Let us check it:
\begin{multline*}
t x_{\alpha_1}(u) t^{-1}=h_{\alpha_1}(r^2) h_{\alpha_3}(r^3) x_{\alpha_1}(u) h_{\alpha_3}(r^3)^{-1} h_{\alpha_1}(r^2)^{-1}=\\
=h_{\alpha_1}(r^2) x_{\alpha_1}(u/ r^3) h_{\alpha_1}(r^2)^{-1}= x_{\alpha_1}(u\cdot r^4/r^3)
=x_{\alpha_1}(u\cdot r),
\end{multline*}
for $\alpha_2$
$$
t x_{\alpha_2}(u) t^{-1}=h_{\alpha_2}(r^2) h_{\alpha_4}(r^4) x_{\alpha_2}(u) h_{\alpha_4}(r^4)^{-1} h_{\alpha_2}(r^2)^{-1}
=x_{\alpha_2}(u\cdot (r^2)^2/r^4)=x_{\alpha_2}(u ),
$$
for $\alpha_3$
\begin{multline*}
t x_{\alpha_3}(u) t^{-1}=h_{\alpha_1}(r^2) h_{\alpha_3}(r^3) h_{\alpha_4}(r^4) x_{\alpha_3}(u) h_{\alpha_4}(r^4)^{-1} h_{\alpha_3}(r^3)^{-1} h_{\alpha_1}(r^2)^{-1}=\\
=x_{\alpha_3}(u\cdot (s^3)^2/(r^2\cdot r^4))=x_{\alpha_3}(u ),
\end{multline*}
for $\alpha_4$
\begin{multline*}
t x_{\alpha_4}(u) t^{-1}=h_{\alpha_2}(r^2) h_{\alpha_3}(r^3) h_{\alpha_4}(r^4) h_{\alpha_5}(r^3) x_{\alpha_4}(u) h_{\alpha_5}(r^3) h_{\alpha_4}(r^4)^{-1} h_{\alpha_3}(r^3)^{-1} h_{\alpha_2}(r^2)^{-1}=\\
=x_{\alpha_4}(u\cdot (s^4)^2/(r^2 \cdot r^3\cdot r^3))=x_{\alpha_4}(u ),
\end{multline*}
for $\alpha_5$
\begin{multline*}
t x_{\alpha_5}(u) t^{-1}=h_{\alpha_4}(r^4) h_{\alpha_5}(r^3) h_{\alpha_6}(r^2) x_{\alpha_5}(u) h_{\alpha_6}(r^2)^{-1} h_{\alpha_5}(r^3)^{-1} h_{\alpha_4}(r^4)^{-1}=\\
=x_{\alpha_5}(u\cdot (r^3)^2/(r^4\cdot r^2))=x_{\alpha_5}(u ),
\end{multline*}
for $\alpha_6$
\begin{multline*}
t x_{\alpha_6}(u) t^{-1}=h_{\alpha_5}(r^3) h_{\alpha_6}(r^2) h_{\alpha_7}(r) x_{\alpha_6}(u) h_{\alpha_7}(r)^{-1} h_{\alpha_6}(r^2)^{-1} h_{\alpha_5}(r^3)^{-1}=\\
=x_{\alpha_6}(u\cdot (r^2)^2/(r^3\cdot r))=x_{\alpha_6}(u ),
\end{multline*}
finally, for $\alpha_7$
$$
t x_{\alpha_7}(u) t^{-1}=h_{\alpha_6}(r^2) h_{\alpha_7}(r) x_{\alpha_7}(u) h_{\alpha_7}(r)^{-1} h_{\alpha_6}(r^2)^{-1}
=x_{\alpha_7}(u\cdot r^2/r^2)=x_{\alpha_7}(u ).
$$

\bigskip

At the end, for $E_8$ we have $S=R$ and
$$
t=h_{\alpha_1}(r^4)h_{\alpha_2}(r^5)h_{\alpha_3}(r^7)h_{\alpha_4}(r^{10})h_{\alpha_5}(r^8)h_{\alpha_6}(r^6)h_{\alpha_7}(r^4)h_{\alpha_8}(r^2).
$$
Check it again:
\begin{multline*}
t x_{\alpha_1}(u) t^{-1}=h_{\alpha_1}(r^4) h_{\alpha_3}(r^7) x_{\alpha_1}(u) h_{\alpha_3}(r^7)^{-1} h_{\alpha_1}(r^4)^{-1}=\\
=h_{\alpha_1}(r^4) x_{\alpha_1}(u/ r^7) h_{\alpha_1}(r^4)^{-1}= x_{\alpha_1}(u\cdot r^8/r^3)
=x_{\alpha_1}(u\cdot r),
\end{multline*}
for $\alpha_2$
$$
t x_{\alpha_2}(u) t^{-1}=h_{\alpha_2}(r^5) h_{\alpha_4}(r^{10}) x_{\alpha_2}(u) h_{\alpha_4}(r^{10})^{-1} h_{\alpha_2}(r^5)^{-1}
=x_{\alpha_2}(u\cdot (r^5)^2/r^{10})=x_{\alpha_2}(u ),
$$
for $\alpha_3$
\begin{multline*}
t x_{\alpha_3}(u) t^{-1}=h_{\alpha_1}(r^4) h_{\alpha_3}(r^7) h_{\alpha_4}(r^{10}) x_{\alpha_3}(u) h_{\alpha_4}(r^{10})^{-1} h_{\alpha_3}(r^7)^{-1} h_{\alpha_1}(r^4)^{-1}=\\
=x_{\alpha_3}(u\cdot (s^7)^2/(r^4\cdot r^{10}))=x_{\alpha_3}(u ),
\end{multline*}
for $\alpha_4$
\begin{multline*}
t x_{\alpha_4}(u) t^{-1}=h_{\alpha_2}(r^5) h_{\alpha_3}(r^7) h_{\alpha_4}(r^{10}) h_{\alpha_5}(r^8) x_{\alpha_4}(u) h_{\alpha_5}(r^8) h_{\alpha_4}(r^{10})^{-1} h_{\alpha_3}(r^7)^{-1} h_{\alpha_2}(r^5)^{-1}=\\
=x_{\alpha_4}(u\cdot (s^{10})^2/(r^5 \cdot r^7\cdot r^8))=x_{\alpha_4}(u ),
\end{multline*}
for $\alpha_5$
\begin{multline*}
t x_{\alpha_5}(u) t^{-1}=h_{\alpha_4}(r^{10}) h_{\alpha_5}(r^8) h_{\alpha_6}(r^6) x_{\alpha_5}(u) h_{\alpha_6}(r^6)^{-1} h_{\alpha_5}(r^8)^{-1} h_{\alpha_{10}}(r^4)^{-1}=\\
=x_{\alpha_5}(u\cdot (r^8)^2/(r^{10}\cdot r^6))=x_{\alpha_5}(u ),
\end{multline*}
for $\alpha_6$
\begin{multline*}
t x_{\alpha_6}(u) t^{-1}=h_{\alpha_5}(r^8) h_{\alpha_6}(r^6) h_{\alpha_7}(r^4) x_{\alpha_6}(u) h_{\alpha_7}(r^4)^{-1} h_{\alpha_6}(r^6)^{-1} h_{\alpha_5}(r^8)^{-1}=\\
=x_{\alpha_6}(u\cdot (r^6)^2/(r^8\cdot r^4))=x_{\alpha_6}(u ),
\end{multline*}
for $\alpha_7$
\begin{multline*}
t x_{\alpha_7}(u) t^{-1}=h_{\alpha_6}(r^6) h_{\alpha_7}(r^4) h_{\alpha_8}(r^2) x_{\alpha_7}(u) h_{\alpha_8}(r^2)^{-1} h_{\alpha_7}(r^4)^{-1} h_{\alpha_6}(r^6)^{-1}=\\
=x_{\alpha_7}(u\cdot (r^4)^2/(r^6\cdot r^2))=x_{\alpha_7}(u ),
\end{multline*}
finally, for $\alpha_8$
$$
t x_{\alpha_8}(u) t^{-1}=h_{\alpha_7}(r^4) h_{\alpha_8}(r^2) x_{\alpha_8}(u) h_{\alpha_8}(r^2)^{-1} h_{\alpha_7}(r^4)^{-1}
=x_{\alpha_8}(u\cdot (r^2)^2/r^4)=x_{\alpha_8}(u ).
$$

\medskip

Therefore, all obtained extensions of~$R$ are found, elements $ t$ are constructed.
\end{proof}

Let us prove now the main theorem (Theorem~\ref{main}).

\begin{proof}
The case when the Chevalley group is elementary adjoint,
evidently follows from Theorems \ref{old} and~\ref{norm}.

Suppose now that we have any other elementary Chevalley group
$E_\pi(\Phi,R)$ and some its automorphism~$\varphi$. Its quotient by the center is isomorphic to the elementary adjoint group $E_{\ad}(\Phi,R)$,
so we have an automorphism $\overline \varphi$ of the group
$E_{\ad}(\Phi,R)$. Such an automorphism is decomposed to
$$
\overline \varphi = \overline \rho \circ
\varphi_{\overline g}\circ \overline \delta,
$$
where $\overline \rho$ is a ring automorphism,  $\overline \delta$ is a graph automorphism, $\varphi_{\overline g}$ is a conjugation with an element
 $\overline g\in G_{\ad}(\Phi,R)$.  Note that an automorphism $\overline \rho$  can be easily changed
to a ring automorphism $\rho$  of the group
$E_\pi(\Phi,R)$ such that $\rho$  on equivalence classes of the group $E_{\pi}(\Phi,R)$ by its center acts in the same way as $\overline \rho$.

Consider now the automorphism $\varphi_{\overline g}$. Note that
$\overline g= \overline t \overline e$, where $\overline t\in
T_{\ad}(\Phi, R)$, $\overline e\in E_{\ad}(\Phi, R)$. For the element
$\overline e$ we can find such $e\in E_{\pi}(\Phi,R)$ that the image
$e$ under factorization by the center is~$\overline e$. The element $\overline t$ has its inverse image $t\in T_{\pi}(\Phi, S)$,
where $S$ is a ring, that is obtained from~$R$ by adding some scalars (see Lemma~\ref{tor}). Also $t$ normalizes the group $E_\pi(\Phi,R)$.
Consider now $g=te\in G_{\pi}(\Phi,S)$. Clear that under factorization of $E_\pi (\Phi, R)$ by the center the automorphism
$\varphi_g$ gives us the automorphism~$\varphi_{\overline g}$.

Now consider the automorphism
$$
\psi=\varphi_{g^{-1}}\circ \rho^{-1}\circ \varphi.
$$
It is an automorphism of the group $E_\pi(\Phi,R)$, under factorization by the center it gives a graph automorphism of the quotient group. Therefore, $\psi$ maps every $x_{\alpha}(t)$, $\alpha\in \Phi$, $t\in R$, into $\lambda_{\alpha,t} \cdot x_{\delta(\alpha)}(\pm t)$, where for any $\alpha$ and $t$ $\lambda_{\alpha,t}$ is a central element of the Chevalley group.

Note that $x_{\alpha+\beta}(t)=[x_\alpha(t),x_\beta(1)]$, so
\begin{multline*}
\lambda_{\alpha+\beta,t}x_{\delta(\alpha+\beta)}(\pm t)=\psi(x_{\alpha+\beta}(t))=[\psi(x_\alpha(t)),\psi(x_\beta(1))]=\\ =
[\lambda_{\alpha,t}x_{\delta(\alpha)}(\pm t), \lambda_{\beta,1} x_{\delta(\beta)}(\pm 1)]=[x_{\delta(\alpha)}(\pm t), x_{\delta(\beta)}
(\pm 1)]=x_{\delta(\alpha+\beta)}(\pm t).
\end{multline*}

Since every root from~$\Phi$ can be represented as a sum of to roots, we have $\lambda_{\alpha,t}=1$ for all $\alpha\in \Phi$, $t\in R$.
Therefore, $\psi$ is a graph automorphism.

Consequently we proved the theorem for all elementary Chevalley groups of types under consideration.
Namely, we proved that every automorphism of $E_{\pi}(\Phi,R)$ is a composition of ring, graph and inner (but not strictly inner) automorphisms.

Now suppose that we have a Chevalley group $G_\pi(\Phi,R)$ and its automorphism~$\varphi$.
Since an elementary group
$E_\pi(\Phi,R)$ is characteristic (commutant) in
$G_\pi(\Phi,R)$, then $\varphi$ is simultaneously an automorphism of the elementary subgroup. On the elementary subgroup it is the composition $\rho\circ \delta\circ \varphi_g$, $g\in
G_{\pi}(\Phi,S)$, also $g=te$, where $e\in E_\pi(\Phi,R)$, $t\in T_\pi (\Phi,S)$. The first two automorphisms are clearly extended to the automorphisms of the whole group $G_{\pi}(\Phi,R)$, and the third is an automorphism of this group, since the torus elements commutes. Then the composition
 $\psi=\varphi_{g^{-1}}\circ \delta^{-1}\circ
\rho^{-1}\circ \varphi$ is an automorphism of
$G_\pi(\Phi,R)$, that acts identically on the elementary subgroup.
Since $G_\pi(\Phi,R)=T_\pi(\Phi,R)\cdot E_{\pi}(\Phi,R)$, we need only to understand how  $\psi$ acts on torus elements. The element
$t^{-1}\varphi(t)\in G_\pi(\Phi,R)$ lies in the center of
$E_\pi(\Phi,)$, therefore, in the center of $G_\pi(\Phi,R)$.
Consequently the automorphism $\psi$ is central.

Therefore, Theorem~\ref{main} is proved for all Chevalley groups under consideration.
\end{proof}

\end{document}